\documentclass[11pt]{amsart}

\usepackage{amsthm, mathrsfs,amssymb,amsmath}
\usepackage{enumerate}

\makeatletter
\@namedef{subjclassname@2010}{%
\textup{2010} Mathematics Subject Classification}
\makeatother

\allowdisplaybreaks

\newtheorem{thm}{Theorem}[section]
\newtheorem{lem}[thm]{Lemma}
\newtheorem{prop}[thm]{Proposition}
\newtheorem{cor}[thm]{Corollary}

\theoremstyle{definition}

\theoremstyle{remark}
\newtheorem{rem}{Remark}

\title[Orlicz modules over homogeneous spaces]{ORLICZ Modules  over Coset Spaces of Compact Subgroups in Locally compact Groups}
\author{Vishvesh Kumar}
\address{Vishvesh Kumar \endgraf School of Mathematical Sciences \endgraf National Institute of Science Education and Research Bhubaneshwar, HBNI  \endgraf At/Po- Jatni, District- Khurda, Odisha- 7520250,  India.} 
\email{vishveshmishra@gmail.com}

\begin{document}

\begin{abstract} Let $H$ be a compact subgroup of a locally compact group $G$ and let $m$ be the normalized $G$-invariant measure on homogeneous space $G/H$ associated with Weil's formula. Let $\varphi$ be a Young function satisfying $\Delta_2$-condition. We introduce the notion of left module action of $L^1(G/H, m)$ on the Orlicz spaces $L^\varphi(G/H, m).$ We also introduce a Banach left $L^1(G/H, m)$-submodule of $L^\varphi(G/H, m).$  

\end{abstract}
\maketitle

\section{Introduction}
 
 The abstract theory of  Banach modules or Banach algebras plays an important role in various branches of Mathematics, for instance, abstract harmonic analysis, representation theory, operator theory; see \cite{Farashahi1,Fe,Fe1,Farabraz} and the references therein. In particular, convolution structure on the Orlicz spaces to  be an Banach algebra or a Banach module over a locally compact group or hypergroup were studied by many researchers \cite{Aghababa,Fe3,Rao2,Kumar,KumarDeb,KumarHaus}.  

In \cite{Kumar},  the author defined and studied the notion of abstract Banach convolution algebra on Orlicz spaces over homogeneous spaces of compact groups. Recently,  Ghaani Farashahi \cite{Farabraz} introduced the notion of abstract Banach convolution function module on the $L^p$-space on  coset spaces of compact subgroups in locally compact groups. The purpose of this article is to define and study a new class of abstract Banach module on Orlicz spaces over coset spaces of compact subgroups in locally compact groups.  Let us remark that Orlicz spaces are genuine generalization of Lebesgue spaces. It is worth mentioning that an appropriate use of  Jensen's inequality \cite[pg. 62]{Rao} plays a key role in this article. 

In the next section, we present some basics of Orlicz spaces and some classical harmonic analysis on a homogeneous space (the space of left cosets) of a locally compact group. Section 3 is devoted to the study of abstract convolution module structure on the Orlicz space $L^\varphi(G/H, m),$ where $H$ is a compact subgroup of a locally compact group $G,$ $m$ is the normalized $G$-invariant measure on the homogeneous space $G/H$ which satisfies Weil's formula and $\varphi$ is a Young function  satisfying $\Delta_2$-condition. In this section, we prove that $L^\varphi(G/H, m)$ is a Banach left $L^1(G/H, m)$-module with respect to a generalized convolution. We also introduce a Banach left $L^1(G/H, m)$-submodule of $L^\varphi(G/H).$

\section{Preliminaries} 
    A non-zero convex function $\varphi : \mathbb{R} \rightarrow [0, \infty]$ is called a {\it Young function} if it is even, left continuous with $\varphi(0)=0$ and $\underset{x \rightarrow \infty}{\lim} \varphi(x)= \infty$. Here we note that every Young function is an integral of a non-decreasing left continuous function \cite[Theorem 1]{Rao}. 
    
    Let $\Omega$ be a locally compact Hausdorff space and $m$ be a positive Radon measure on $\Omega.$ Denote the space of all equivalence classes of $m$-measurable functions on $\Omega$ by $L^0(\Omega).$ A Young function $\varphi$ satisfies {\it $\Delta_2$-condition} if there exist a constant $C>0$ and $x_0 >0$ such that $ \varphi(2x) \leq C\varphi(x)$  for all $x \geq x_0 $ if $m(\Omega)<\infty$ and $ \varphi(2x) \leq C\varphi(x)$  for all $x \geq 0 $ otherwise.
     Given a Young function $\varphi$, the {\it modular function} $\rho_\varphi : L^0(\Omega) \rightarrow \mathbb{R}$ is defined by $\rho_\varphi(f):=\int_\Omega \varphi(|f|)\, dm.$ We always assume that Young function $\varphi$ satisfies $\Delta_2$-condition. For a given Young function $\varphi,$ the {\it Orlicz space} $L^\varphi(\Omega, m),$ in short $L^\varphi(\Omega),$ is defined by 
  $$L^{\varphi}(\Omega) := \left\lbrace f \in L^0(\Omega) : \rho_\varphi(af)< \infty~~ \mbox{for ~~some}\,\, a>0  \right\rbrace.$$ 
  Then the Orlicz space is a Banach space with respect to the norm $\|\cdot \|_\varphi^0$ on $L^\varphi(\Omega)$ called {\it Luxemburg norm} or {\it gauge norm} which is defined by
  $$\|f \|_\varphi^0 := \mbox{inf} \left\lbrace k>0: \int_G \varphi \left(\frac{|f|}{k} \right) \,dm \leq 1 \right\rbrace.$$   
  
 If $\varphi(x)= |x|^p$ $1 \leq p <\infty$ then $L^\varphi(\Omega)$ is usual $L^p$-spaces, $1 \leq p <\infty.$ An example of Young function which satisfies $\Delta_2$-condition  and gives an Orlicz space other than $L^p$-spaces is given by 
 $\varphi(x)= (e+|x|)\, \text{log}(e+|x|)-e.$

We denote the space of all continuous functions  on $\Omega$ with compact support by $\mathcal{C}_c(\Omega).$ It is well known that if $\varphi$ satisfies $\Delta_2$-condition then $\mathcal{C}_c(\Omega)$ is a dense subspace of $L^\varphi(\Omega).$ If $A$ is a measurable subset of $\Omega$ such that $0<m(A)< \infty,$ then we have the {\it Jensen's Inequality} 
\begin{eqnarray*}
\varphi\left(\frac{\int_A f \,dm}{m(A)} \right) \leq \frac{\int_A \varphi(f)\, dm}{m(A)}.  
\end{eqnarray*}
We make use of the above inequality several times in this article. In this paper, we also employ the notations of the author in \cite{Farashahi,Farabraz}.

  For a locally compact group $G$ with the Haar measure $dx$ and $f,g \in L^\varphi(G),$ define convolution $*_G$ on $L^\varphi(G,m)$ by $$f*_Gg(x)= \int_G f(y) \,g(y^{-1}x)\, dx \,\,\,\,\, (x \in G).$$  It is well-known that $(L^\varphi(G), \|\cdot\|_\varphi^0)$ is a Banach algebra  with respect to the convolution product $*_{G}$ (see \cite{Hudzik}) that is, 
   $$\|f *_G g\|_\varPhi^0 \leq \|f\|_\varphi^0\,  \|g\|_\varphi^0$$
   for all $f, g \in L^\varphi(G).$
   Also, if $f \in L^1(G)$ and $f \in L^\varphi(G)$ then the above convolution define a module action of $L^1(G)$ on $L^\varphi(G)$ which makes $L^\varphi(K)$ a Banach left $L^1(K)$-module, that is, 
      \begin{equation} \label{orliczmodule}
              \|f *_G g\|_\varphi^0 \leq \|f\|_1\, \|g\|_\varphi^0
      \end{equation}
      for all $f \in L^1(G)$ and $g \in L^\varphi(G)$ (see \cite{KumarDeb,Raomodule}).

  Let $H$ be a compact subgroup of a locally compact group $G$ with the normalized Haar measure $dh.$ The left coset space $G/H$ can be seen as a homogeneous space with respect to the action of $G$ on $G/H$ given by left multiplication. The canonical surjection $q: G \rightarrow G/H$ is given by $q(x)= xH.$ Define $T_H(f)(xH)= \int_H f(xh)\,dh,$ then
  \begin{equation*}
  \mathcal{C}_c(G/H) =  \left\lbrace T_H(f): f \in \mathcal{C}_c(G)\right\rbrace 
\end{equation*}  

 The homogeneous space $G/H$ has a unique normalized $G$-invariant positive Radon measure $m$ that satisfies the Weil's formula 
\begin{equation} \label{weil}
\int_{G/H} T_H(f)(xH) \, dm(xH) = \int_G f(x)\, dx,
\end{equation}
and hence $\|T_H(f)\|_{L^1(G/H, m)} \leq \|f\|_{L^1}(G)$ for all  $f \in L^1(G).$ For more details on harmonic analysis on homogeneous spaces of locally compact groups see \cite{Farashahi,Farashahi1,Farashahi2,Farashahi3,Farashahi4,Reiter,Farabraz}.  


  \section{Orlicz modules over Coset Spaces of compact subgroups of locally compact groups} 
  
  Throughout this section, we assume that the Young function $\varphi$ satisfies $\Delta_2$- condition, $G$ is a locally compact group with a Haar measure $dx$ and $H$ is a compact subgroup of $G$ with the normalized Haar measure $dh.$ 
  It is also assumed that the homogeneous space $G/H$ has the normalized $G$-invariant measure $m$ satisfying the Weil's formula. In this section, we show that the space $L^\varphi(G/H)$ becomes a Banach left $L^1(G/H, m)$-module with respect to the convolution on $\mathcal{C}_c(G/K)$ defined in \cite{Farabraz}. We also define a subspace of $L^\varphi(G/H)$ and show that this subspace is a Banach left $L^1(G/H, m)$-submodule of $L^\varphi(G/H, m).$
  
  We begin this section with following result.
  
   \begin{thm}  \label{bddT_H} 
  Let $G$ be a locally compact group and let $H$ be a compact subgroup of $G.$ Let $m$ be the normalized $G$-invariant measure on the coset space $G/H.$   Then the linear map $T_H: \mathcal{C}_c(G) \rightarrow \mathcal{C}_c(G/H)$ satisfies 
  \begin{equation} \label{bdd}
  \|T_H(f)\|_{L^\varphi(G/H, m)}^0 \leq \|f\|_{L^\varphi(G)}^0
\end{equation} 
for all $f \in \mathcal{C}_c(G).$ Further, if the Young function $\varphi$ satisfies $\Delta_2$-condition. $T_H$ can be uniquely extended to a linear map from $L^\varphi(G)$ onto $L^\varphi(G/H, m).$  
  \end{thm}
  \begin{proof}
  For $f \in \mathcal{C}_c(G)$ and $k >0,$ by using Weil's formula and Jensen's inequality, we have 
  \begin{eqnarray*}
  \rho_\varphi \left( \frac{T_H(f)}{k}\right) &=& \int_{G/H} \varphi \left( \frac{|T_H(f)|(xH)}{k}\right) \, dm(xH) \\ &=& \int_{G/H} \varphi \left( \left| \int_H \frac{f(xh)}{k} \, dh \right| \right) dm(xH) \\& \leq & \int_{G/H} \varphi \left(  \int_H \frac{|f(xh)|}{k} \, dh  \right) dm(xH)\\ & \leq & \int_{G/H} \int_H \varphi \left( \frac{|f(xh)|}{k} \right) dh \, dm(xH)\\ &=&   \int_{G/H} \left( \int_H \varphi \left( \frac{|f|}{k} \right)(xh)\, dh \right) \, dm(xH) \\ &=& \int_{G/H} T_H\left( \varphi \left( \frac{|f|}{k}\right)\right)(xH)\, dm(xH) \\&=& \int_G \varphi \left( \frac{|f|}{k} \right) dx  = \rho_\varphi\left(\frac{f}{k}\right). 
  \end{eqnarray*}
  Now, \begin{eqnarray*}
  \|f\|_{L^\varphi(G)}^0 &=& \inf \{k>0 : \rho_\varphi\left( \frac{f}{k}\right) \leq 1\}\\  &\geq & \inf\{k>0 : \rho_\varphi\left( \frac{T_H(f)}{k}\right) \leq 1\}=\|T_H(f)\|_{L^\varphi(G/H,m)}^0. 
  \end{eqnarray*}
  Therefore, we get $\|T_H(f)\|_{L^\varphi(G/H,m)}^0 \leq \|f\|_{L^\varphi(G)}^0 $ for all $f \in \mathcal{C}_c(G).$
  
  Since $\phi$ is $\Delta_2$-regular, $\mathcal{C}_c(G)$ and $\mathcal{C}_c(G/H)$ are dense in $L^\varphi(G)$ and $L^\varphi(G/H, m)$ respectively. Therefore, we can extend $T_H$ to a bounded linear map from $L^\varphi(G)$ onto $L^\varphi(G/H, m).$ We denote this extension of $T_H$ again by $T_H$ and it satisfies \eqref{bdd}. 
  \end{proof}
  
\begin{prop} \label{f_q} Let $G$ be a locally compact group and let $H$ be a compact subgroup of $G.$ Let $m$ be the normalized $G$-invariant measure on the coset space $G/H.$ Suppose that $\varphi$ satisfies $\Delta_2$-condition. 
If $f \in L^\varphi(G/H, m)$ and $f_q:= f \circ q$ then we have $f_q \in L^\varphi(G)$ with 
\begin{equation}
\|f_q\|_{L^\varphi(G)}^0 = \|f\|_{L^\varphi(G/H, m)}^0.
\end{equation}  
  \end{prop}
 \begin{proof} 
   For $f \in L^\varphi(G/H, m)$ and $k>0,$ by Weil's formula and the fact that $H$ is compact, we have 
   \begin{eqnarray*}
   \rho_\varphi \left( \frac{f_q}{k} \right)&=& \int_G \varphi \left( \frac{|f_q|(x)}{k}\right) dx\\ &=& \int_{G/H} T_H \left( \varphi \left( \frac{|f_q|}{k}\right)\right)(xH)\, dm(xH)\\ &=& \int_{G/H} \left(\int_H \varphi \left(\frac{|f_q|(xh)}{k} \right) dh \right) dm(xH)\\
&=& \int_{G/H} \left(\int_H \varphi \left(\frac{|f(xhH)|}{k} \right) dh \right) dm(xH) \\ &=& \int_{G/H} \left(\int_H \varphi \left(\frac{|f(xH)|}{k} \right) dh \right) dm(xH) \\ &=& \left( \int_{G/H} \varphi \left( \frac{|f|(xH)}{k} \right) dm(xH) \right) \left( \int_H \, dh \right) = \rho_\varphi \left(\frac{f}{k} \right). 
 \end{eqnarray*}
 
 Therefore $ \rho_\varphi \left( \frac{f_q}{k} \right)= \rho_\varphi \left(\frac{f}{k} \right)$ and consequently, we get  $ \|f_q\|_{L^\varphi(G)}^0 = \|f\|_{L^\varphi(G/H, m)}^0.$ \end{proof}
\begin{rem} Note that $T_H(f_q)=f$ and therefore, it is clear from above corollary that for $f_q \in L^\varphi(G),$ the equality in Theorem \ref{bddT_H} holds.
\end{rem}
  
Here we set certain terminologies for further use. For any continuous function $f$ define the left translation by  $L_hf(\cdot)=f(h^{-1} (\cdot))$ and the right translation by  $R_hf=f((\cdot) h).$ Let $G$ be a locally compact group and let $H$ be a compact subgroup of $G.$ We set  
\begin{eqnarray*}
\mathcal{C}_c(G:H)&=& \{f \in \mathcal{C}_c(G): R_hf=f \,\forall\, h \in H\}, \\ A(G:H)&=& \{f \in \mathcal{C}_c(G) : L_hf=f \,\forall\, h \in H\} \\
\mbox{and }A(G/H) &=& \{g  \in \mathcal{C}_c(G/H) : L_hg=g\, \forall\, h \in H \}.\end{eqnarray*} 

For a Young function $\varphi \in \Delta_2,$ define $$A^\varphi(G:H)= \{f \in L^\varphi(G) : L_hf =f \,\forall\, h \in H\},$$ and also $$A^\varphi(G/H,m)= \{g \in L^\varphi(G/H, m) : L_hg =g \,\forall\, h \in H\}.$$  

Note that $A^\varphi(G/H, m)$ is the topological closure of $A(G/H)$ in $L^\varphi(G/H, m)$ and therefore it is a closed subspace of $L^\varphi(G/H, m).$ Similarly $A^\varphi(G:H)$ is closed subspace of $L^\varphi(G).$ It is known that $T_H$ maps $\mathcal{C}_c(G:H)$ and $A(G:H)$ onto $\mathcal{C}_c(G/H)$  and $A(G/H)$ respectively (see \cite[Proposition 4.2]{Farabraz}).

Since $A(G:H)$ is dense in $L^\varphi(G:H)$ and $A(G/H)$ is dense in $A^\varphi(G/H),$ the next lemma follows from the continuity of $T_H.$
\begin{lem} Let $G$ be a locally compact group and let $H$ be a compact subgroup of $G.$ Let $m$ be the normalized $G$-invariant measure on the coset space $G/H.$ Suppose that $\varphi$ satisfies $\Delta_2$-condition. Then the mapping $T_H$ maps $L^\varphi(G:H)$ onto $A^\varphi(G/H, m).$  
\end{lem}

For $g \in C_c(G/H),$ note that the mapping $xH\mapsto \int_H g(hxH) \, dh$ is in $\mathcal{C}_c(G/H).$ Define $J:\mathcal{C}_c(G/H)\rightarrow\mathcal{C}_c(G/H)$ as  $$Jg(xH)= \int_H g(hxH) \, dh.$$ It is clear that $J$ is a linear operator. In addition, the following theorem says that the norm of $J$ is bounded operator with the norm bounded by one.

\begin{thm} \label{bbdJ} For $f \in \mathcal{C}_c(G/H),$ we have $$\|Jf\|_{L^\varphi(G/H, m)} \leq \|f\|_{L^\varphi(G/H, m)}.$$  
\end{thm}
  \begin{proof}
  For  $f \in \mathcal{C}_c(G/H)$ we have, 
  \begin{eqnarray*}
  \rho_\varphi\left( \frac{Jf}{k}\right) &=& \int_{G/H} \varphi \left( \frac{|Jf|(xH)}{k} \right) dm(xH)\\ &=& \int_{G/H} \varphi \left( \left| \frac{1}{k} \int_H f(hxH)\, dh\right| \right) dm(xH)\\ &\leq & \int_{G/H} \int_H \varphi \left(\frac{|f|(hxH)}{k} \right) \, dh \,  dm(xH)  \\ &=& \int_H  \left( \int_{G/H} \varphi \left( \frac{|f|(hxH)}{k}\right)\, dm(xH)  \right) dh \\ &=&  \int_H  \left( \int_{G/H} \varphi \left( \frac{|f|(xH)}{k}\right)\, dm(xH)  \right) dh \\ &=& \int_H  \left( \int_{G/H} \varphi \left( \frac{|f|}{k}\right)(xH)\, dm(xH) \right) dh = \rho_\varphi \left(\frac{f}{k}\right).
\end{eqnarray*}
Consequently, we get $\|Jf\|_{L^\varphi(G/H, m)} \leq \|f\|_{L^\varphi(G/H, m)}$ for all $f \in \mathcal{C}_c(G/H).$
\end{proof}

It is shown in \cite[Theorem 4.5 (2)]{Farabraz} that $J$ maps $\mathcal{C}_c(G/H)$ onto $A(G/H).$ Now, the following corollary is immediate. 
\begin{cor} \label{g}  Let $G$ be a locally compact group and let $H$ be a compact subgroup of $G.$ Let $m$ be the normalized $G$-invariant measure on the coset space $G/H.$ Suppose that $\varphi$ satisfies $\Delta_2$-condition.
The bounded linear map $J: \mathcal{C}_c(G/H) \rightarrow A(G/H)$ can be uniquely extended to a bounded linear map $J_\varphi: L^\varphi(G/H, m) \rightarrow A^\varphi(G/H, m) $ which satisfies $$\|J_\varphi f\|_{L^\varphi(G/H, m)} \leq \|f\|_{L^\varphi(G/H, m)}.$$ 
Further, the linear operator $J_\varphi: L^\varphi(G/H, m) \rightarrow A^\varphi(G/H, m) $ is an onto map.
\end{cor}

\begin{proof}
The extension of the map $J$ from $\mathcal{C}_c(G/H)$ to $L^\varphi(G/H)$ follows from Theorem \ref{bbdJ} and the density of $\mathcal{C}_c(G/H,m)$ and $A(G/H)$ in $L^\varphi(G/H,m)$ and  $A^\varphi(G/H, m),$ respectively. Further, for any $f \in L^\varphi(G/H)$ and $z \in H,$ we have 
\begin{eqnarray*}
L_z(Jf)(xH)= Jf (z^{-1}xH) = \int_H f(hz^{-1}xH) \,dh = \int_H f(hxH) \,dh = Jf(xH).
\end{eqnarray*}
This shows that $Jf \in A^\varphi(G/H,m).$ Now we prove second part of the corollary. For any $f \in A^\varphi(G/H),$ we get,
\begin{eqnarray*}
Jf(xH)= \int_H f(hxH) \,dh = \int_H f(xH)\,dh = f(xH),
\end{eqnarray*}
for all $x \in G.$ Therefore, $Jf=f.$ Hence $J_\varphi:  L^\varphi(G/H,m) \rightarrow A^\varphi(G/H,m)$ is a onto map.
\end{proof}
\begin{rem} \label{r2} It can seen in the proof of Corollary \ref{g} above that $J|_{A^\varphi(G/H,m)}= I_{A^\varphi(G/H,m)}.$
\end{rem}

Now, we are ready to define the convolution product `$*_{G/H}$' on $\mathcal{C}_c(G/H)$ same as given in \cite{Farabraz} as follows: let $G$ be a compact group, $H$ a closed subgroup and let $m$ be the normalized $G$-invariant measure on $G/H.$ For $f,g \in \mathcal{C}_c(G/H),$ the convolution $f*_{G/H}g: G/H \rightarrow \mathbb{C}$ is given by 
\begin{equation} \label{3}
f*_{G/H}g(xH)= \int_{G/H} f(yH)Jg(y^{-1}xH)\, dm(yH),
\end{equation}
for all $xH \in G/H.$  The convolution product `$*_{G/H}$' has the following properties similar to the usual convolution in $\mathcal{C}_c(G)$ (see \cite[Proposition 4.10]{Farabraz}).
\begin{itemize}
	\item[(i)] For any $f,g \in \mathcal{C}_c(G/H),$ $(f,g) \mapsto f*_{G/H}g$ is a bilinear map from $\mathcal{C}_c(G/H) \times \mathcal{C}_c(G/H)$ to $\mathcal{C}_c(G/H)$ and $(\mathcal{C}_c(G/H), *_{G/H})$ is an algebra.
	\item[(ii)] $f*_{G/H}g = T_H(f_q*_Gg_q)$ and $(f *_{G/H}g)_q= f_q*_G g_q,$ where `$*_G$' is the usual convolution in $C(G).$
	\item[(iii)]  $L_x(f*_{G/H}g)= (L_x f)*_{G/H}g.$
\end{itemize}

The following result says that $\mathcal{C}_c(G/H)$ is a normed algebra with respect to the norm $\|\cdot\|_{L^\varphi(G/H, m)}^0.$ 
\begin{lem} \label{C_c} If $f,g \in \mathcal{C}_c(G/H),$ then
	\begin{equation}
	\|f*_{G/H}g\|_{L^\varphi(G/H, m)}^0 \leq \|f\|_{L^1(G/H, m)} \|g\|_{L^\varphi(G/H, m)}^0.
	\end{equation} 
\end{lem}
\begin{proof}
	Let $f,g \in \mathcal{C}_c(G/H).$ Using Proposition \ref{f_q} we get 
	\begin{eqnarray*}
		\|f*_{G/H}g\|_{L^\varphi(G/H, m)}^0 = \|(f*_{G/H}g)_q\|_{L^\varphi(G)}^0 = \|f_q*_{G}g_q\|_{L^\varphi(G)}^0
	\end{eqnarray*}
	Since $L^\varphi(G)$ is a $L^1(G)$-module then by Proposition \ref{f_q} we get
	\begin{align*}
	\|f*_{G/H}g\|_{L^\varphi(G/H, m)}^0  &\leq  \|f_q\|_{L^1(G)} \|g_q\|_{L^\varphi(G)}^0 =  \|f\|_{L^1(G/H, m)} \|g\|_{L^\varphi(G/H, m)}^0. \qedhere
	\end{align*} \end{proof} 

\begin{thm} \label{7}   Let $G$ be a locally compact group and let $H$ be a compact subgroup of $G.$ Let $m$ be the normalized $G$-invariant measure on the coset space $G/H.$ Suppose that $\varphi$ satisfies $\Delta_2$-condition. Then
	the convolution $*_{G/H} : \mathcal{C}_c(G/H) \times \mathcal{C}_c(G/H) \rightarrow \mathcal{C}_c(G/H)$ given by  \eqref{3} can be extended to a convolution $*_{G/H}^\varphi: L^1(G/H, m) \times L^\varphi(G/H, m) \rightarrow L^\varphi(G/H, m)$ such that $L^\varphi(G/H,m)$ is a Banach left  $L^1(G/H, m)$-module with respect to this extended convolution.
\end{thm} 
\begin{proof} Let $f \in L^1(G/H, m)$ and $ g \in L^\varphi(G/H,m).$ 
	Since $C_c(G/H)$ is dense in $L^1(G/H, m)$ and in  $L^\varphi(G/H, m)$ as $\varphi \in \Delta_2,$ there exist $\{f_n\}$ and $\{g_n\}$ in $\mathcal{C}_c(G/H)$ such that $f_n \rightarrow f$ and $g_n \rightarrow g$ as $n \rightarrow \infty.$ Now, define  $$f*_{G/H}^\varphi g = \lim_{n \rightarrow \infty} f_n*_{G/H}g_n.$$ Note that $*_{G/H}^\varphi: L^1(G/H,m) \times L^\varphi(G/H,m) \rightarrow L^\varphi(G/H,m) $ is well-defined. By Lemma \ref{C_c} we have  $$ \|f *_{G/H}^\varphi g \|_{L^\varphi(G/H,m)}^0 \leq \|f\|_{L^1(G/H,m)} \|g\|_{L^\varphi(G/H, m)}^0. $$ Thus $L^\varphi(G/H, m)$ is a Banach left $L^1(G/H, m)$-module with respect to the extended convolution.
\end{proof}

The above theorem claims the existence of convolution product $*_{G/H}^\varphi$ but it does not reveal any explicit formula for the convolution product. The following corollary fulfils this objective whose proof is a consequence of the fact that $\mathcal{C}_c(G/H)$ is dense $L^\varphi(G/H, m).$

\begin{cor} \label{8} Let $G$ be a locally compact group and let $H$ be a compact subgroup of $G.$ Let $m$ be the normalized $G$-invariant measure on the coset space $G/H.$ Suppose that $\varphi$ satisfies $\Delta_2$-condition. If $f \in L^1(G/H, m)$ and $ g \in L^\varphi(G/H, m),$ then the convolution $*_{G/H}^\varphi$ is  given by \begin{equation} \label{5}
	(f*_{G/H}^\varphi g) (xH)= \int_{G/H} f(yH)J_\varphi g(y^{-1}xH)\, dm(yH)
	\end{equation}
	for all $xH \in G/H.$ \end{cor}
In the next corollary we present a Banach left $L^1(G/H, m)$-submodule of $L^\varphi(G/H,m)$ whose proof is a routine check. 
\begin{cor}
Let $G$ be a locally compact group and let $H$ be a compact subgroup of $G.$ Let $m$ be the normalized $G$-invariant measure on the coset space $G/H.$ Suppose that $\varphi$ satisfies $\Delta_2$-condition. Then the space $A^\varphi(G/H,m)$ is a Banach left $L^1(G/H, m)$-submodule of $L^\varphi(G/H,m).$
\end{cor}

\subsection*{Acknowledgements}
The author thanks Prof. V. Muruganandam for his support and encouragement.

\end{document}